\documentclass[11pt]{amsart}
\usepackage{graphicx} % Required for inserting images
\usepackage{amsmath, amssymb}

\usepackage{mathptmx}
\usepackage{amsmath}
\usepackage{amscd}
\usepackage{amssymb}
\usepackage{amsthm}
\usepackage{xspace}
\usepackage{hyperref}
\usepackage[capitalize]{cleveref}
\usepackage{xcolor}
\usepackage{mathpazo}
\usepackage[marginratio=1:1,height=8.5in,width=6.5in,tmargin=1.25in]{geometry}
\usepackage{todonotes}
\presetkeys%
    {todonotes}%
    {inline,backgroundcolor=yellow}{}
\usepackage{mathtools}

\DeclarePairedDelimiter{\ceil}{\lceil}{\rceil}
\allowdisplaybreaks
\theoremstyle{plain}
\newtheorem{thm}{Theorem}[section]

\newtheorem{prop}[thm]{Proposition}
\newtheorem{lemma}[thm]{Lemma}
\newtheorem{ques}{Question}

\theoremstyle{definition}

\newtheorem{remark}{Remark}

\DeclareMathOperator{\Mod}{Mod} 
\DeclareMathOperator{\SL}{SL} \DeclareMathOperator{\PSL}{PSL}
 \DeclareMathOperator{\PGL}{PGL}

\DeclareMathOperator{\Jac}{Jac}

\DeclareMathOperator{\tr}{tr}
\newcommand{\norm}[1]{\left\| #1 \right\|}

\begin{document}
\title{On the Jacobian of the Douady-Earle extension}
\author{Chris Connell$^\dagger$, Yuping Ruan, Shi Wang$^{\ddagger}$}

\address{Department of Mathematics,
Indiana University, Bloomington, IN 47405, USA}
\email{connell@iu.edu}
\address{Department of Mathematics,
Northwestern University, Evanston, IL 60208, USA}
\email{ruanyp@northwestern.edu}
\address{Institute of Mathematical Sciences,
ShanghaiTech University, Shanghai 201210, China}
\email{wangshi@shanghaitech.edu.cn}
\thanks{$^\dagger$ Supported in part by Simons Foundation grant \#965245 and National Science Foundation grant DMS-2514510}
\thanks{$^\ddagger$ Supported in part by NSFC grant \#12301085}
\maketitle

\begin{abstract}
     Given an isotopy class between two closed hyperbolic surfaces, the Douady--Earle extension provides a unique analytic diffeomorphism representative. In this paper we investigate the Jacobian of the Douady--Earle extension map $F$. We prove that $|\Jac F| \equiv 1$ precisely when $F$ is an isometry. Moreover, we construct a sequence of hyperbolic surfaces $\{\Sigma_i\}$ together with a fixed domain surface $\Sigma_0$ for which the Douady--Earle extension maps $F_i:\Sigma_0\to\Sigma_i$ satisfy $\max_{x\in\Sigma_0} \Jac F_i \to +\infty$. 
\end{abstract}

\section{Introduction}

For each $g \geq 2$, let $S_g$ denote a closed, oriented surface of genus $g$. 
Teichm\"uller theory concerns the global structure of the Teichm\"uller space 
$\mathcal{T}_g$, the space of all marked hyperbolic structures on $S_g$. 
Beginning with the foundational work of Fricke, Klein, Teichm\"uller, Fenchel, and Nielsen 
\cite{Fricke,Teich1939,Teich1944,FN02}, it is known that $\mathcal{T}_g$ is 
homeomorphic to $\mathbb{R}^{6g-6}$. Since then, the theory has developed 
extensive connections with complex geometry, low-dimensional topology, geometric 
group theory, and dynamical systems. In particular, numerous metric structures on $\mathcal T_g$ have been investigated, including the Teichm\"uller metric \cite{Teich1939}, the Thurston metric \cite{Th1986,Pa1988}, the Weil-Petersson metric \cite{Weil1958,Ah1961}, and many others \cite{Yau1978,ChernYau1975,LiuSunYau2004,LiuSunYau2009,Royden1971,McMullen2000}. 

The Teichm\"uller metric is a complete Finsler metric that measures the best quasi-conformal dilation constant between two Riemann surfaces. The Thurston metric is an asymmetric Finsler metric, defined similarly using the best Lipschitz constant between two hyperbolic surfaces. The Weil-Petersson metric is a noncomplete, Kahler metric that gives rise to a finite volume metric on the moduli space.

In this work, we investigate the Jacobian of the Douady--Earle extension map $F$. 
Our motivation is to understand how the volume distortion produced by $F$ 
relates to the geometry of Teichm\"uller space. 

Given two marked hyperbolic surfaces $X_1:=(f_1, \Sigma_1)$ and $X_2:=(f_2, \Sigma_2)$ in the Teichm\"uller space where $f_i:S_g\to \Sigma_i$ ($i=1,2$) are the "marking" homeomorphisms, following the work of Douady-Earle \cite{DE86}, there exists a unique ``barycenter map'' $F:\Sigma_1\to \Sigma_2$ within the same isotopy class of $f_2\circ f_1^{-1}$ known as the Douady-Earle extension. It is shown in \cite{DE86} that $F$ is an analytic diffeomorphism between two hyperbolic surfaces. Note that the unique harmonic map in the same isotopy class of $f_2\circ f_1^{-1}$ is also known to be a diffeomorphism (\cite{Schoen-Yau, Sampson}), but these maps are not always the same (\cite{Li12, Jiang-Liu-Yao}).

We associate to this unique Douady-Earle extension map the following asymmetric function
\[J(X_1,X_2):=\log \left(\max_{x\in \Sigma_1}|\Jac F(x)|\right).\]
Since $F$ preserves the total area, there always exists a point $x\in \Sigma_1$ with $|\Jac F(x_1)|\geq 1$, thus $J$ is a non-negative function on $\mathcal T_g\times \mathcal T_g$. Moreover, any mapping class $[\sigma]\in \Mod(S_g)$ acts on $\mathcal T_g$ via 
$[\sigma](f,\Sigma)=(f\circ \sigma^{-1},\Sigma)$, so the Douady--Earle extension map defined for the pair $[\sigma](X_1)$ and $[\sigma](X_2)$ remains the same. This shows that $J$ is invariant under the diagonal action of the mapping class group $\Mod(S_g)$. The main purpose of this paper is to show the following two results.

\begin{thm}\label{thm:rigidity}
    For any $g\geq 2$, and any pair $X,Y\in \mathcal T_g$, we have $J(X,Y)=0$ if and only if $X=Y$.
\end{thm}
\begin{thm}\label{thm:unbounded}
    For any $g\geq 2$, the function $J$ is unbounded on $\mathcal T_g\times \mathcal T_g$.
\end{thm}

\begin{remark}
    We note that \cref{thm:unbounded} corrects a misstatement in \cite{Li12} which claimed that the $J$-function is bounded. 
\end{remark}

One motivation for studying this $J$ function comes from the higher dimensional vision. In \cite{BCG95}, Besson-Courtois-Gallot generalized the Douady-Earle's construction to a large family of maps between Riemannian manifolds. By implementing a delicate sharp estimate on the Jacobian of their barycenter map, they proved the entropy rigidity conjecture for rank one symmetric spaces. (See also \cite{Ruan}.) In the special case where the starting map is a homotopy equivalence between two hyperbolic manifolds of dimension $\geq3$, they showed that the Jacobian of the barycenter map is always bounded above by $1$, or equivalently that the analog of the above defined $J$ value is zero, and further that the barycenter map must be an isometry. In particular, this gives a new proof of the Mostow rigidity. When the target manifold is an irreducible higher rank symmetric space excluding type $\SL_3(\mathbb R)$, it is shown \cite{CF03,CF03err} that the Jacobian is uniformly bounded. Nevertheless, the case of hyperbolic surfaces is missing from the literature. In stark contrast to the situation in higher dimensions, we show in the proof of \cref{thm:unbounded} that along a family of hyperbolic surface metrics constructed by collapsing a closed geodesic in a ``symmetric way'', the Jacobian of the Douady-Earle extension tends to infinity. It is plausible that this $J$ function measures, in a certain sense, the distance between two (marked) hyperbolic structures on an underlying surface. We end this section with the following question.

\begin{ques}
    Is this $J$ function an asymmetric distance function on $\mathcal T_g$? How does it compare to the known natural metrics?
\end{ques}

\subsection*{Acknowledgements.}
The third author thanks Huiping Pan and Junming Zhang for helpful discussions.

\section{Douady-Earle's barycenter extension map}
For $S_g$ a closed oriented surface of genus $g$, the Teichm\"uller space $\mathcal T_g$ consists of all hyperbolic structures on $S_g$ up to isotopy. More precisely, we have
\[\mathcal T_g:=\{(f,\Sigma)\;|\; f:S_g\to \Sigma\}/\sim,\]
where $\Sigma$ is a hyperbolic surface and $f$ is a homeomorphism between $S_g$ and $\Sigma$, and the equivalence relation is given by
\[(f_1,\Sigma_1)\sim (f_2,\Sigma_2)\textrm{ if and only if } f_2\circ f_1^{-1} \textrm{ is isotopic to an isometry.}\]
Each pair $(f,\Sigma)$ gives a hyperbolic structure on $S_g$ by pulling back the metric, and we call the map $f$ a marking. Each point $(f,\Sigma)\in \mathcal T_g$ corresponds to a representation $\pi_1(S_g)\to \PSL_2(\mathbb R)$ by composing the monodromy representation $\pi_1(\Sigma)\to \PSL_2(\mathbb R)$ with $f_*$, and the equivalence relation $\sim$ corresponds to the $\PGL_2(\mathbb R)$ conjugation equivalence on the representation space.

Given a pair of points on the Teichm\"uller space $X_1:=(f_1,\Sigma_1)$ and $X_2:=(f_2,\Sigma_2)$, the map $f_2\circ f_1^{-1}$ induces a group isomorphism $\rho:\pi_1(\Sigma_1)\to \pi_1(\Sigma_2)$. Fix two basepoints $O_1$ and $O_2$ on the universal cover $\mathbb H^2$ of $\Sigma_1$ and $\Sigma_2$. Since the actions of $\pi_1(\Sigma_i)$ on $\mathbb H^2$ are cocompact, they give rise to quasi-isometries between $\mathbb H^2$ and $\pi_1(\Sigma_i)$ via the orbit maps. Thus, it determines a unique quasi-isometry between $\mathbb H^2$, which is given by
\[\mathbb H^2\sim \pi_1(\Sigma_1)\cdot O_1\sim \pi_1(\Sigma_1)\sim_{\rho} \pi_1(\Sigma_2)\sim \pi_1(\Sigma_2)\cdot O_2\sim\mathbb H^2.\]
This induces a boundary homeomorphism 
\begin{equation}\label{eq:boundary-map}
    \varphi\colon\partial_\infty\mathbb H^2\to \partial_\infty\mathbb H^2.
\end{equation}
It is clear that $\varphi$ does not depend on the choice of basepoints $O_1, O_2$.

For any discrete isometry subgroup $\Gamma$ acting on $\mathbb H^2$, there exists \cite{Patterson1976} a family of Radon measures $\{\mu_x\;|\;x\in \mathbb H^2\}$ called the Patterson-Sullivan measures which satisfies the following properties:
\begin{enumerate}
    \item[$\bullet$] $\{\mu_x\}$ supports on $\Lambda(\Gamma)\subset \partial_\infty \mathbb H^2$ the limit set of $\Gamma$.
    \item[$\bullet$] $\{\mu_x\}$ is $\Gamma$-equivariant, that is, $\mu_{\gamma x}=\gamma_*\mu_x$ for any $x\in \mathbb H^2$ and $\gamma\in \Gamma$.
    \item[$\bullet$] $\{\mu_x\}$ is pairwise equivalent, and for any $x,y\in \mathbb H^2$, and any $\theta\in \Lambda(\Gamma)$ we have explicit Radon-Nykodym derivatives
    \begin{equation}\label{eq:RN-derivative}
        \frac{d\mu_x}{d\mu_y}(\theta)=e^{-\delta B_\theta(x,y)}.
    \end{equation}
    Here $\delta$ is the critical exponent of $\Gamma$ and 
 $$B_\theta(x,y):=\lim_{t\to +\infty}[d(x,c_{y\theta}(t))-t)]$$ denotes the Busemann function on $\mathbb H^2$ where $c_{y\theta}$ is the unit speed geodesic ray connecting $y$ to $\theta$.
\end{enumerate}
When $\Gamma$ is a lattice in $\PSL_2(\mathbb R)$, we know $\delta=1$ and that $\{\mu_x\}$ is in fact $\PSL_2(\mathbb R)$-equivariant. Indeed, $\{\mu_x\}$ coincides with the harmonic measures. Moreover, $\{\mu_x\}$ is the same as the visual measures upto a multiplicative scalar. To be specific, we let $ds_x$ be the Haar measure of $T^1_x\mathbb{H}^2$. Consider the visual map $\alpha_x:T^1_x\mathbb{H}^2\to\partial_\infty\mathbb{H}^2$ at $x$, which is a diffeomorphism defined as $\alpha_x(v)=v(-\infty)$, where $v(-\infty)$ denotes the endpoint of the geodesic ray with initial vector $-v$. Then we have 
\begin{align}\label{eqn:visual}
    \alpha_x^*\mu_x=\lambda ds_x\text{ for some }\lambda>0,\text{ and }\alpha_x^{-1}(\theta)=\nabla B_\theta (\cdot,y)|_x\text{ for any }\theta\in\partial_\infty\mathbb{H}^2\text{ and for any }y\in\mathbb{H}^2.
\end{align}

On the other hand, let $\mathcal{M}(\partial_\infty\mathbb H^2)$ denote the space of Radon measures on $\partial_\infty\mathbb H^2$. For any finite, nonzero and non-atomic measure $\nu\in \mathcal{M}(\partial_\infty\mathbb H^2)$, we can define the barycenter of $\nu$ by the formula
\[\operatorname{bar}(\nu):=x_{\nu}\in \mathbb H^2\]
such that at $x_{\nu}$ the function
\begin{equation}\label{eq:average-Buse}
    \mathcal B_\nu(x):=\int_{\partial_\infty\mathbb H^2} B_\theta(x,O)d\nu(\theta)
\end{equation}
reaches its unique minimum. Note that the definition of $\operatorname{bar}(\nu)$ is independent on the choice of basepoint $O$ and we often abbreviate $B_\theta(x,O)$ with $B_{\theta,x}$ when the basepoint either is chosen or does not matter.

We are now ready to define the Douady-Earle extension. For a pair of points on the Teichm\"uller space $X_1:=(f_1,\Sigma_1)$ and $X_2:=(f_2,\Sigma_2)$, we see from \eqref{eq:boundary-map} that there is a boundary homeomorphism $\varphi \colon\partial_\infty\mathbb H^2\to \partial_\infty\mathbb H^2$. Now we define $\widetilde{F}\colon \mathbb H^2\to \mathbb H^2$ to be
\begin{equation}\label{eq:barycenter}
    \widetilde{F}(x):=\operatorname{bar}(\varphi_*(\mu_x)),
\end{equation}
where $\varphi_*(\cdot)$ denotes the push-forward measures. 
Since $\{\mu_x\}$ is $\pi_1(\Sigma_1)$-equivariant, $\varphi$ is $\rho$-equivariant, and $\operatorname{bar}(\cdot)$ is  $\pi_1(\Sigma_2)$-equivariant, we see that $\widetilde F$ is $\rho$-equivariant, hence it descents to a map $F\colon \Sigma_1\to \Sigma_2$, which we call the Douady-Earle extension map (or the barycenter map). It is shown in \cite[Theorem 1]{DE86} that $F$ is an analytic diffeomorphism.

We now follow the treatment of \cite{BCG95} to obtain further properties of $F$. For any $x\in \mathbb H^2=\widetilde\Sigma_1$ and $y=\widetilde F(x)\in \mathbb H^2=\widetilde\Sigma_2$. Since at $y$ the function $\mathcal B_\nu$ (See \eqref{eq:average-Buse}) achieves the unique minimum, we have $d\mathcal B_{\nu,y}=0$ for $\nu=\varphi_*\mu_x$. Therefore, we have the following $1$-form equation
\begin{equation}\label{eq:1-form}
    0=\int_{\partial_\infty \mathbb H^2}dB_{\xi,y}(\cdot)(\varphi_*d\mu_x)(\xi)=\int_{\partial_\infty \mathbb H^2}dB_{\varphi(\theta),y}(\cdot)d\mu_x(\theta),
\end{equation}
where the second equality pulls back the integral by $\varphi$. Differentiating \eqref{eq:1-form} on both sides about the $x$ variable, and in view of \eqref{eq:RN-derivative} ($\delta=1$), we obtain the following equation for $(1,1)$-tensors
\begin{equation}\label{eq:2-form}
    \int_{\partial \mathbb H^2} \nabla dB_{\varphi(\theta),y}(d\widetilde F_x(\cdot))d\mu_x(\theta)=\int_{\partial \mathbb H^2} dB_{\theta,x}(\cdot)\nabla B_{\varphi(\theta),y} d\mu_x(\theta).
\end{equation}
This holds for any $y=\widetilde{F}(x)$.

\section{Proof of Theorem \ref{thm:rigidity}}
The ``if part'' is clear since if $\Sigma_1, \Sigma_2$ are isometric, then the isometry satisfies the defining equation $\eqref{eq:barycenter}$ of the Douady-Earle extension. 

Now assume $J(X,Y)=0$, that is, the Douady-Earle extension map $F$ between the pairs $X=(\Sigma_1, f_1)$, $Y=(\Sigma_2,f_2)$ satisfies $|\Jac F|(x)=1$ for all $x\in \Sigma_1$. We claim that there exists $x_0\in \Sigma_1$ such that $dF_{x_0}$ is an local isometry (possibly orientation reversing). Otherwise, at every point $x\in \Sigma_1$, we have $\norm{dF}>1$ with respect to the Riemannian inner product at $x$ and $y=F(x)$. By the Implicit Function Theorem, it follows that there is a smooth line field  $\mathcal{L}$ on $\Sigma_1$ such that at any point and either unit vector $e\in \mathcal{L}$, $|dF(e)|=\norm{dF}$  (i.e. $\mathcal{L}$ is the most expanding line). However the existence of a smooth line field $\mathcal{L}$ contradicts the fact that the Euler characteristic satisfies $\chi(\Sigma_1)\neq 0$. 

Thus, there is a point $x_0$ at which $dF_{x_0}$ is a local isometry. We denote $y_0=F(x_0)$. It follows that there exist orthonormal frames $\{e_1, e_2\}$ at $T_{x_0}\Sigma_1$ and $\{\overline{e_1}, \overline{e_2}\}$ at $T_{y_0}\Sigma_2$ such that $dF_{x_0}(e_1)=\overline{e_1}$ and $dF_{x_0}(e_2)=\overline{e_2}$. It is possible that $F$ is orientation reversing so that the frames $\{e_1, e_2\}$ and $\{\overline{e_1}, \overline{e_2}\}$ have reversed orientations. Lifting to the univeral cover $\mathbb H^2$, and we abuse the same notation $y_0=\widetilde F(x_0)$. Substituting into $\eqref{eq:2-form}$, we obtain that
\[ \int_{\partial \mathbb H^2} \nabla dB_{\varphi(\theta),y_0}(\overline{e_i})d\mu_{x_0}(\theta)=\int_{\partial \mathbb H^2} dB_{\theta,x_0}(e_i)\nabla B_{\varphi(\theta),y_0} d\mu_{x_0}(\theta).\]
For simplicity we normalize so that $\|\mu_{x_0}\|=1$. Now take the trace on both sides and note that $\tr \nabla dB_{\varphi(\theta),y_0}=\Delta B_{\varphi(\theta),y_0}\equiv 1$ in $\mathbb H^2$. We have
\begin{align*}
    1&=\sum_{i=1}^2\left\langle\int_{\partial \mathbb H^2} \nabla dB_{\varphi(\theta),y_0}(\overline{e_i})d\mu_{x_0}(\theta), \overline {e_i}\right\rangle\\
    &=\sum_{i=1}^2 \int_{\partial \mathbb H^2} dB_{\theta,x_0}(e_i)\langle \nabla B_{\varphi(\theta),y_0},\overline {e_i}\rangle d\mu_{x_0}(\theta)\\
    &=\int_{\partial \mathbb H^2}\left(\sum_{i=1}^2 dB_{\theta,x_0}(e_i)\cdot dB_{\varphi(\theta),y_0}(\overline {e_i})\right) d\mu_{x_0}(\theta)\\
    &\leq \int_{\partial \mathbb H^2}\frac{1}2\sum_{i=1}^2 \left(dB^2_{\theta,x_0}(e_i)+ dB^2_{\varphi(\theta),y_0}(\overline {e_i})\right) d\mu_{x_0}(\theta)=1,
\end{align*}
where the above inequality uses the mean inequality and the last equality uses the fact that $\tr dB_{\theta,x_0}^2=\tr dB^2_{\varphi(\theta),y_0}\equiv 1$ in $\mathbb H^2$. This forces the equality on the above inequality. Recall from \eqref{eqn:visual} that $\{\mu_x\}$ are all equivalent to the Lebesgue measure. Thus, for any $i=1,2$ and for Lebesgue a.e. $\theta\in\partial_\infty\mathbb{H}^2$, we have
\begin{align}\label{eqn:local.isom.a.e.}
dB_{\theta,x_0}(e_i)=dB_{\varphi(\theta),y_0}(\overline {e_i}).
\end{align}
Since both $\{e_1, e_2\}$ at $T_{x_0}\mathbb H^2$ and $\{\overline{e_1}, \overline{e_2}\}$ at $T_{y_0}\mathbb H^2$ are orthonormal, we consider a unique isometry $G:\mathbb H^2\to \mathbb H^2$ such that $G(x_0)=y_0$, $dG(e_1)=\overline{e_1}$ and $dG(e_2)=\overline{e_2}$. Then \eqref{eqn:visual} and \eqref{eqn:local.isom.a.e.} imply that $\varphi$ coincide with the boundary map induced by $G$. Notice that $G$ is the unique Douady-Earle extension map of its induced boundary map. This implies that $\widetilde{F}=G$ and hence $F$ is an isometry. In particular, $X=Y\in \mathcal T_g$.

\section{An explicit deformation}\label{sec:deformation}
In order to describe an explicit deformation of hyperbolic structures on the Teichm\"uller space, we briefly recall the Fenchel-Nielson coordinates \cite{FN02}. (Also see \cite[Section 10.6]{FM12} for a nice exposition) Given $S$ a closed oriented surface of genus $g$, and a decomposition into $(2g-2)$ pairs of pants by cutting along $(3g-3)$ disjoint simple closed curves on $S$. The lengths of these curves will uniquely determine a hyperbolic structure (with totally geodesic boundaries) for each pants, which consists of $(3g-3)$ positive length parameters $\ell_1,\dots, \ell_{3g-3}$. Along with these length parameters, there are also $(3g-3)$ twist parameters $\theta_1,\dots, \theta_{3g-3}$ (taking values in $\mathbb R$) that determine how the pairs of pants are glued together. The Fricke-Klein Theorem \cite{Fricke} states that there is a homeomorphism
\[\Phi\colon \mathcal T_g\to \mathbb R_+^{3g-3}\times \mathbb R^{3g-3}\]
given explicitly by the length and twisting parameters $(\ell_1,\dots,\ell_{3g-3};\theta_1,\dots,\theta_{3g-3})$.

Now we choose the first pants--one of the two pants towards the end which consists of only two simple closed curves, and with loss of generality, we denote it by $P_1=(\gamma_1,\gamma_2)$, where $\gamma_1$ is the simply closed curved shared by the two boundary components of the pants (see Figure \ref{fig:deformation}). Denote $\beta_1$ the seam on $P_1$ that transversely intersects $\gamma_1$. For $\epsilon\in (0,1]$, we choose a smooth curve $X_\epsilon:=(\Sigma_\epsilon, f_\epsilon)$ on $\mathcal T_g$ given by 
\[X_\epsilon=\Phi^{-1}(\epsilon,1,\dots,1;0,\dots,0).\]
Intuitively, this curve represents a deformation of hyperbolic structures on $S$ in such a way that, as $\epsilon\to 0^+$, it collapses the $\gamma_1$ curve in the most ``symmetric'' way. Since the twist parameter on $\gamma_1$ remains zero, we see that the closed geodesics in the class of $\gamma_1$ and $\beta_1$ always intersect orthogonally. We denote the intersection point by $p_{\epsilon}$. 

For the remaining section, we will study some of the properties on the family of the Douady-Earle extension maps $F_\epsilon$ from $X_1$ to $X_\epsilon$. We first describe certain symmetries on $\Sigma_\epsilon$.

\begin{figure}[ht]\centering
       \includegraphics[scale=0.65]{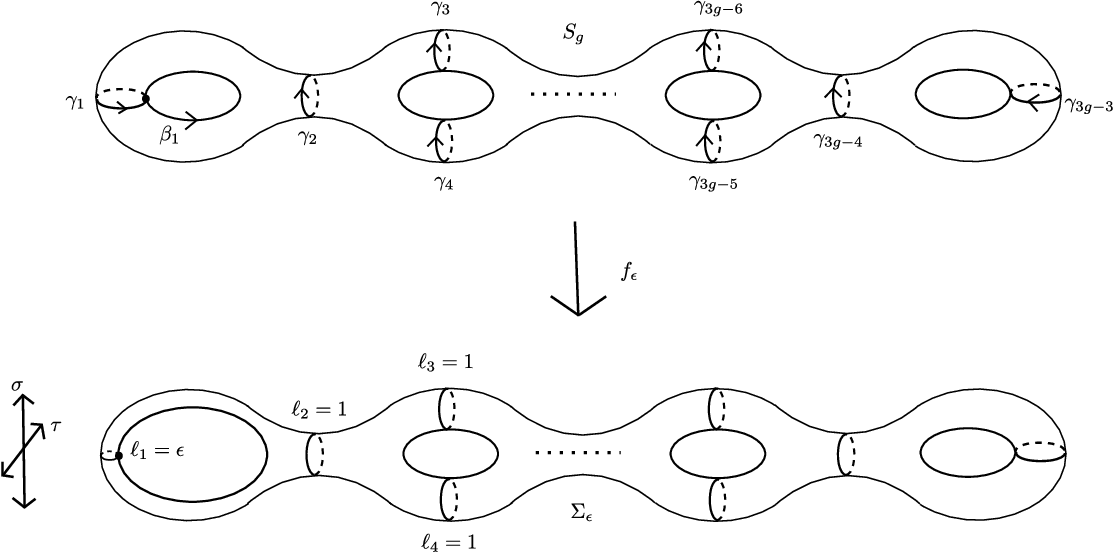}
       \caption{The deformation $X_\epsilon$}
       \label{fig:deformation}
\end{figure}

\subsection*{$(\mathbb Z/2\times \mathbb Z/2)$-symmetry:}
 There are two commutative involutive isometries $\sigma^{(\epsilon)},\tau^{(\epsilon)}$ on each $\Sigma_\epsilon$ which we describe below:
\begin{enumerate}
    \item[$\bullet$] The $\sigma$-symmetry: this corresponds to the up-and-down symmetry (see Figure \ref{fig:deformation}) by fixing the closed geodesics in the homotopy class of $\gamma_1$ and $\gamma_{3g-3}$ (later we simply call $\gamma_1$-geodesic and so on), stabilizing each $\gamma_{3i-1}$-geodesic (for $i=1,\dots, (g-1)$), and permuting $\gamma_{3i}$-geodesics with $\gamma_{3i+1}$-geodesics (for $i=1,\dots, (g-2)$). This involution can be defined first on each individual pants since the length parameters are the same ($\ell_{3i}=\ell_{3i+1}$), and then it extends to a global isometry on $\Sigma_\epsilon$ since the twist parameters are all zero.
    \item[$\bullet$] The $\tau$-symmetry: this corresponds to the front-and-back symmetry as in Figure \ref{fig:deformation}. On each pair of pants, it fixes the seams and permutes the two hyperbolic hexagons isometrically. Similarly, since all twist parameters are zero, it extends globally to an isometry on $\Sigma_\epsilon$. Under $\tau$, all $\gamma_i$-geodesics are stabilized, and all seams of the pants (including $\beta_1$ in particular) are fixed.
    \item[$\bullet$] From the definition, it is clear that $\sigma$ commutes with $\tau$.
    \item[$\bullet$] The families $\sigma^{(\epsilon)}$ and $\tau^{(\epsilon)}$ on $\Sigma_\epsilon$ are the conjugations by $f_\epsilon$ of the respective involutions $\sigma$ and $\tau$ on $S_g$. In particular, the isotopy class determined by the natural identification $\rho_\epsilon:=(f_\epsilon\circ f_1^{-1})_*:\pi_1(\Sigma_1,p_1)\to \pi_1(\Sigma_\epsilon,p_\epsilon)$ is both $(\sigma^{(1)},\sigma^{(\epsilon)})$ and $(\tau^{(1)},\tau^{(\epsilon)})$ equivariant.
\end{enumerate}
It is also convenient to describe $\sigma$ and $\tau$ when lifted to the universal cover. Suppose we choose the disk model of $\mathbb H^2:=\widetilde{\Sigma_\epsilon}$ by choosing a lift $O_\epsilon$ of $p_\epsilon$ to be the origin, the lift of $\beta_1$-geodesic (through $O_\epsilon$) the $x$-axis, and the lift of $\gamma_1$-geodesic (through $O_\epsilon$) the $y$-axis. These axes divide $\mathbb H^2$ into four isometric regions we call \emph{quadrants}. (See Figure \ref{fig:quadrant}.) We also denote the ideal boundary of quadrants in $\partial_\infty\mathbb H^2$ by the same term -- quadrants.  Moreover, $\tau$ and $\sigma$ simply lift to the reflection isometries about the $x$ and $y$-axes, respectively. For simplicity, we still denote these lifts again by $\tau$ and $\sigma$. 

\begin{figure}[ht]\centering
       \includegraphics[scale=0.65]{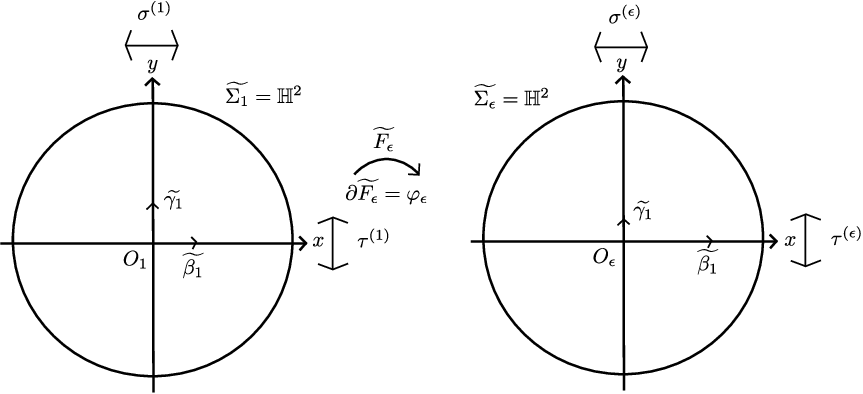}
       \caption{Douady-Earle extension map $\widetilde{F_\epsilon}$}
       \label{fig:quadrant}
\end{figure}

In the proposition below we consider the map $\varphi_\epsilon:\partial_\infty\mathbb H^2\to  \partial_\infty\mathbb H^2$
induced from $\rho_\epsilon$ as in \eqref{eq:boundary-map}, except that we replace $\Sigma_2$ by $\Sigma_\epsilon$.

\begin{prop}\label{prop:boundary-symmetry}
    The boundary map $\varphi_\epsilon$ preserves quadrants, and is both $(\sigma^{(1)},\sigma^{(\epsilon)})$ and $(\tau^{(1)},\tau^{(\epsilon)})$-equivariant.
\end{prop}
\begin{proof}
For each $\epsilon$, the $\mathbb Z$-subgroup of $\pi_1(\Sigma_\epsilon,p_\epsilon)$ corresponding to the $\beta_1$-curve acts on the $x$-axes by translation, hence their forwardward/backward limit points must be identified via the boundary map $\varphi_\epsilon$, that is, we have $\varphi_\epsilon(\pm 1,0)=(\pm 1, 0)$ in the disk model. Similarly, by taking the $\mathbb Z$-subgroup corresponding to the $\gamma_1$-curve, we can show $\varphi_\epsilon(0,\pm 1)=(0, \pm 1)$. Moreover, since the boundary map is a homeomorphism, it must preserve each quadrant. 

For any $\theta\in \partial_\infty\mathbb H^2=\partial_\infty\widetilde{\Sigma_1}$, it can be approximated by a sequence of orbit points. We assume $g_iO_1\to \theta$ where $g_i\in \pi_1(\Sigma_1,p_1)$. Suppose $\varphi_\epsilon(\theta)=\xi\in \partial_\infty\widetilde{\Sigma_\epsilon}$, then $\rho_\epsilon(g_i)O_\epsilon\to \xi$ by definition. Let $\sigma_*^{(\epsilon)}: \pi_1(\Sigma_\epsilon, p_\epsilon) \to \pi_1(\Sigma_\epsilon, p_\epsilon)$ be the automorphism induced by $\sigma^{(\epsilon)}$. Applying the $\sigma_*^{(1)}$-involution, we obtain that
\[\sigma_*^{(1)}(g_i)O_1\to \sigma^{(1)}(\theta),\]
and
\[\sigma_*^{(\epsilon)}\left(\rho_\epsilon(g_i)\right)O_\epsilon\to\sigma^{(\epsilon)}(\varphi_\epsilon(\theta))=\sigma^{(\epsilon)}(\xi).\]
Use the $(\sigma^{(1)},\sigma^{(\epsilon)})$-equivariance of $\rho_\epsilon$, we see that
\[\varphi_\epsilon(\sigma^{(1)}(\theta))=\lim_{i}[\rho_\epsilon(\sigma_*^{(1)}(g_i))O_\epsilon]=\lim_{i}[\sigma_*^{(\epsilon)}\left(\rho_\epsilon(g_i)\right)O_\epsilon]=\sigma^{(\epsilon)}(\varphi_\epsilon(\theta)).\]
This shows $\varphi_\epsilon$ is $(\sigma^{(1)},\sigma^{(\epsilon)})$-equivariant. Similarly, it is $(\tau^{(1)},\tau^{(\epsilon)})$-equivariant.
\end{proof}

We are now ready to discuss the corresponding symmetry of the Douady-Earle extension maps $F_\epsilon:\Sigma_1\to \Sigma_\epsilon$.

\begin{prop}\label{prop:DE-symmetry}
    On the universal cover, the Douady-Earle extension map $\widetilde{F_\epsilon}:\mathbb H^2\to \mathbb H^2$ sends the (positive/negative) $x$-axis to the (positive/negative) $x$-axis and sends the (positive/negative) $y$-axis to the (positive/negative) $y$-axis. In particular, it sends $O_1$ to $O_\epsilon$. Thus, when it descends to the quotient manifolds, the map $F\colon \Sigma_1\to \Sigma_\epsilon$ sends the $\gamma_1$-geodesic on $\Sigma_1$ to the $\gamma_1$-geodesic on $\Sigma_\epsilon$, and sends the $\beta_1$-geodesic on $\Sigma_1$ to the $\beta_1$-geodesic on $\Sigma_\epsilon$. In particular, it sends $p_1$ to $p_\epsilon$.
\end{prop}
\begin{proof}
    Suppose $x_0$ lies on the $x$-axis on the base space $\widetilde{\Sigma_1}\cong\mathbb H^2$, then the Patterson-Sullivan measure $\mu_{x_0}$ (which coincides with the visual measure at $x_0$) is certainly $\tau_*^{(1)}$-invariant. Since the boundary map $\varphi_\epsilon$ is $(\tau^{(1)},\tau^{(\epsilon)})$-equivariant by Proposition \ref{prop:boundary-symmetry}, we see that the push-forward measure $(\varphi_\epsilon)_*\mu_{x_0}$ is $\tau^{(\epsilon)}$-invariant. Thus, its barycenter is also $\tau^{(\epsilon)}$-invariant. Then it must lie on the $x$-axis. 
    
    Moreover, if $x_0$ lies on the positive $x$-axis, then for any $\theta$ lying on the first and fourth quadrants, the density of $d\mu_{x_0}$ at $\theta$ is strictly bigger than the density at $\sigma(\theta)$. On the other hand, since the boundary map preserves the quadrants and is $(\sigma^{(1)},\sigma^{(\epsilon)})$-equivariant, the push-forward measure $(\varphi_\epsilon)_*\mu_{x_0}$ also weighs more at any $\theta$ on the first and fourth quadrants than at $\sigma^{(\epsilon)}(\theta)$ on the second and third quadrants. Therefore, in view of \eqref{eq:1-form}, it's barycenter must also lie on the positive $x$-axis. 
    
    The other part follows in a similar way.
\end{proof}

\section{Proof of Theorem \ref{thm:unbounded}}
In this section we introduce the notation $\lesssim$ to indicate that the left hand side is bounded by a uniform constant multiple of the right hand side. Similarly $\gtrsim$ the reverse bound.
We prove the following stronger statement. As a result, Theorem \ref{thm:unbounded} follows.
\begin{thm}
    Under the deformation described in Section \ref{sec:deformation}, the Douady-Earle extension map $F_\epsilon:\Sigma_1\to \Sigma_\epsilon$ satisfies that
    \[\Jac F_\epsilon(p_1) \gtrsim \frac{1}{\epsilon\ln^2\epsilon}.\]
    In particular, as $\epsilon\to 0^+$, we have $\Jac F_\epsilon\to +\infty$.
\end{thm}
\begin{proof}
    We use the disk model and build up the coordinate systems as in Section \ref{sec:deformation} by setting $O_\epsilon$ the origin, the lift of $\beta_1$-geodesic (through $O_\epsilon$) the $x$-axis, and the lift of $\gamma_1$-geodesic (through $O_\epsilon$) the $y$-axis. Let $\{e_1,e_2\}$ be the standard base axial orthonormal frame at $T_{O_1}\mathbb H^2$ and $\{\overline{e_1},\overline{e_2}\}$ be the target axial orthonormal frame at $T_{O_\epsilon}\mathbb H^2$. By Proposition \ref{prop:DE-symmetry}, we know that at $O_1$, the differential map $d\widetilde F$ sends $e_1$ to the direction of $\overline{e_1}$ and $e_2$ to the direction of $\overline{e_2}$. We assume
    \begin{equation}\label{eq:eigenvalue}
        d\widetilde F(e_1)=\lambda_1\overline{e_1}\;\textrm{ and } \;d\widetilde F(e_2)=\lambda_2\overline{e_2}\;\textrm{ for some }\lambda_1,\lambda_2>0\;\textrm{ depending on }\epsilon.
    \end{equation}
    Then we have $\Jac \widetilde F_\epsilon=\lambda_1\cdot \lambda_2$. 
    The proof of the theorem follows immediately from the following lemma.
\end{proof}
    
\begin{lemma} We have aymptotic bounds
    \[\lambda_1\gtrsim \frac{1}{\epsilon^2\ln^2\epsilon}\;\textrm{ and }\;\lambda_2\gtrsim \epsilon.\]
\end{lemma}
\begin{proof}    
    We substitute $e_1, e_2$ into the 2-form equation of the Douady-Earle extension \eqref{eq:2-form} and get for $i=1,2$
    \[\int_{\partial_\infty \mathbb H^2} \nabla dB_{\varphi(\theta),O_\epsilon}(d\widetilde F(e_i))d\mu(\theta)=\int_{\partial_\infty \mathbb H^2} dB_{\theta,O_1}(e_i)\nabla B_{\varphi(\theta),O_\epsilon} d\mu(\theta),\]
    where $d\mu$ is the visual (Lebesgue) measure at $O_1$. Take the inner product of the above equation with $\overline e_i$ and use \eqref{eq:eigenvalue}, we obtain that
    \[\lambda_i=\frac{\int_{\partial_\infty \mathbb H^2} dB_{\theta,O_1}(e_i) dB_{\varphi_\epsilon(\theta),O_\epsilon}(\overline e_i) d\mu(\theta)}{\int_{\partial_\infty \mathbb H^2} \nabla dB_{\varphi_\epsilon(\theta),O_\epsilon}(\overline e_i, \overline e_i)d\mu(\theta)},\;\textrm{ for } i=1,2.\]
    If we choose $\theta\in [0,2\pi)$ the angle parameter and use the $\mathbb Z/2\times \mathbb Z/2$ symmetry of $\mu$ and $\varphi_\epsilon$, then we can simplify further to
    \begin{equation}\label{eq:la1}
        \lambda_1=\frac{\int_{0}^{\pi/2}\cos\theta\cdot \cos\varphi_\epsilon(\theta)d\theta}{\int_{0}^{\pi/2}\sin^2\varphi_\epsilon(\theta)d\theta},
    \end{equation}
    and 
    \begin{equation}\label{eq:la2}
        \lambda_2=\frac{\int_{0}^{\pi/2}\sin\theta\cdot \sin\varphi_\epsilon(\theta)d\theta}{\int_{0}^{\pi/2}\cos^2\varphi_\epsilon(\theta)d\theta}.
    \end{equation}
For the convenience, we introduce the following function $$\Theta:\mathbb R^+\to (0,\pi/2)$$ with the property that the unique bi-infinite geodesic, which intersects the positive $y$-axis orthogonally and is (hyperbolic) distance $s$ away from the origin, intersects the boundary circle at angle $\Theta(s)$ and $\pi-\Theta(s)$. (See Figure \ref{fig:Theta}.) Using the hyperbolic law of cosines, we see that
\begin{equation}
    \cosh(s)=\frac{1}{\cos(\Theta(s))}.
\end{equation}
Therefore, 
$$\lim_{s\to 0^+}\frac{\Theta(s)}{s}=1=\lim_{s\to+\infty}\frac{\pi/2-\Theta(s)}{2e^{-s}}.$$

\begin{figure}[ht]\centering
       \includegraphics[scale=0.65]{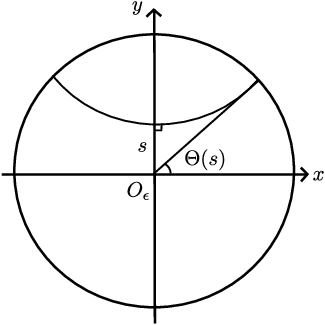}
       \caption{The $\Theta$ function}
       \label{fig:Theta}
\end{figure}

Denote $a\in \pi_1(S_g)$ the element which corresponds to the $\gamma_1$ curve, then for any $k\in \mathbb Z_+$, the element $(f_\epsilon)_*(a^k)\in \pi_1(\Sigma_\epsilon,p_\epsilon)$ acts on its universal cover isometrically such that it translates the $x$-axis to the  geodesic which is distance $(k\cdot \epsilon)$ away from the origin, and meanwhile intersecting the positive $y$-axis orthogonally. This means that the boundary map $\varphi_\epsilon$ sends the point at angle $\Theta(k)$ to the point at angle $\Theta(k\cdot\epsilon)$ for any $k\in \mathbb Z_+$ since $\varphi_\epsilon$ is $\rho_\epsilon$-equivariant. Moreover, since $\varphi_\epsilon$ is a homeomorphism, we know that it is strictly increasing. In particular, it sends $(0,\Theta(k))$ homeomorphically to $(0,\Theta(k\epsilon))$ for each $k\in \mathbb Z_+$. (See Figure \ref{fig:map}) Now we are ready to estimate $\lambda_1$ and $\lambda_2$. 

\begin{figure}\centering
       \includegraphics[scale=0.65]{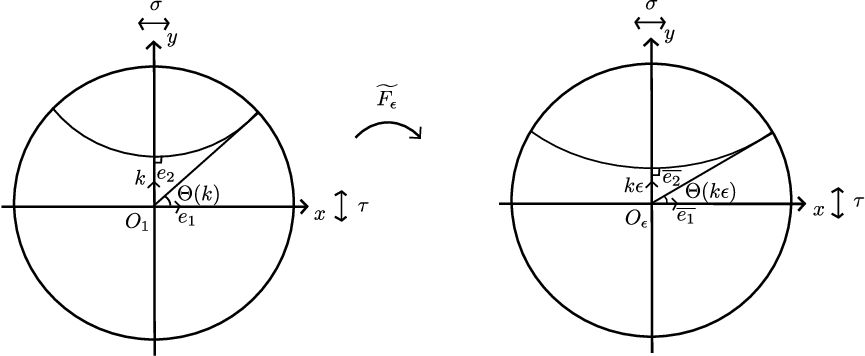}
       \caption{The map $F_\epsilon$}
       \label{fig:map}
\end{figure}

For $\lambda_1$, we choose $k=\ceil{-2\ln(\epsilon)}$. Then the numerator of \eqref{eq:la1} has lower bounds
\begin{align*}
\begin{split}
    \int_{0}^{\pi/2}\cos\theta\cdot \cos\varphi_\epsilon(\theta)d\theta&\geq \int_{0}^{\Theta(1)}\cos\theta\cdot \cos\varphi_\epsilon(\theta)d\theta\\
    &\geq \cos^2(\Theta(1))\cdot \Theta(1)\\
    &=O(1),
\end{split}
\end{align*}
where the second inequality uses $\varphi_\epsilon(\Theta(1))=\Theta(\epsilon)\leq \Theta(1)$. The denominator of \eqref{eq:la1} has upper bounds
\begin{align*}
\int_{0}^{\pi/2}\sin^2\varphi_\epsilon(\theta)d\theta&=\int_{0}^{\Theta(k)}\sin^2\varphi_\epsilon(\theta)d\theta+\int_{\Theta(k)}^{\pi/2}\sin^2\varphi_\epsilon(\theta)d\theta\\
    &\leq \sin^2\Theta(k\epsilon)\cdot \frac{\pi}{2}+1\cdot \left(\frac{\pi}{2}-\Theta(k)\right)\\
    &\lesssim \epsilon^2\ln^2\epsilon+ e^{-k}\\
    &\lesssim \epsilon^2\ln^2\epsilon.
\end{align*}
Thus, we obtain
\[\lambda_1\gtrsim  \frac{1}{\epsilon^2\ln^2\epsilon}.\]

For $\lambda_2$, the numerator of \ref{eq:la2} has lower bounds
\begin{align*}
    \int_{0}^{\pi/2}\sin\theta\cdot \sin\varphi_\epsilon(\theta)d\theta &\geq \int_{\Theta(1)}^{\pi/2}\sin\theta\cdot \sin\varphi_\epsilon(\theta)d\theta\\
    &\geq \sin\Theta(1)\cdot \sin\Theta(\epsilon)\cdot \left(\frac{\pi}{2}-\Theta(1)\right)\\
    &\gtrsim \epsilon.
\end{align*}
The denominator of \eqref{eq:la2} has upper bounds
\begin{align*}
    \int_{0}^{\pi/2}\cos^2\varphi_\epsilon(\theta)d\theta\leq \frac{\pi}2.
\end{align*}
Thus, we have
\[\lambda_2\gtrsim \epsilon.\]
\end{proof}

\bibliography{bcgsurface.bib}
\bibliographystyle{alpha}

\end{document}